\documentclass[11pt]{35}
\usepackage{amsthm, amsmath, amssymb, amsbsy, amsfonts, latexsym, euscript}
\usepackage{calrsfs} 
\usepackage{graphicx}

\DeclareMathAlphabet{\varmathbb}{U}{pxsyb}{m}{n}

\def\leq{\leqslant}

\def\geq{\geqslant}
\def\phi{\varphi}

\def\bar{\overline}

\def\kappa{\varkappa}

\newcommand{\D}{\mathrm{d}\kern0.2pt}%

\newcommand{\E}{\mathrm{e}\kern0.2pt} 
\newcommand{\ii}{\kern0.05em\mathrm{i}\kern0.05em}

\newcommand{\RR}{\mathbb{R}}%

\newtheorem{theorem}{\bf \indent Theorem}[section]
\newtheorem{proposition}{\bf \indent Proposition}[section]
\newtheorem{lemma}{\bf \indent Lemma}[section]
\newtheorem{corollary}{\bf \indent Corollary}[section]

\theoremstyle{remark}

\newtheorem{definition}{\bf\indent Definition}[section] 

\newtheorem{conjecture}{\bf \indent Conjecture}[section]

\numberwithin{equation}{section}

\begin{document}

\noindent {\Large \bf On analytic characterization of convex sets in
$\mathbf{\RR^m}$ (a survey)}

\vskip5mm

{\bf Nikolay Kuznetsov}

\vskip-2pt {\small Laboratory for Mathematical Modelling of Wave Phenomena}
\vskip-4pt {\small Institute for Problems in Mechanical Engineering} \vskip-4pt
{\small Russian Academy of Sciences} \vskip-4pt
{\small V.O., Bol'shoy pr. 61, St Petersburg 199178, Russia \vskip-4pt {\small
nikolay.g.kuznetsov@gmail.com}

\vskip7mm

\parbox{144mm} {\noindent In the first part of this note, we review results
concerning analytic characterization of convexity for planar sets. The second part
is devoted to results valid for arbitrary $m \geq 2$.}

\vskip6mm

{\centering \section{Introduction} }

\noindent The notion of a convex set was introduced by Hermann Minkowski at the
close of the 19th century. In 1897--1903, he published four papers treating
different aspects of convex sets (see \cite[p.~224]{Val} for the corresponding
references), and his work initiated the modern theory of convexity. Minkowski's
definition (it is given in his first paper \cite[p.~198]{M}, where these sets are
studied on their own), is equivalent to the common, modern one based on an intrinsic
property of convex sets (cf. \cite[Definition~1.3]{Val}):

\begin{definition}
A set $S \subset \RR^m$ is convex if with every two points $x, y$ belonging to $S$
the entire segment joining $x$ and $y$ lies in $S$.
\end{definition}

It is amazing that Motzkin's theorem that characterizes the convexity in another
natural way---in terms of the nearest point property---appeared only in 1935.

\begin{theorem}[Motzkin, \cite{Mo}; see also \cite{Val}, Theorem 7.8]
Let $S$ be a closed set in $\RR^m$. This set is convex if and only if a unique
nearest point in $S$ corresponds to every point in $\RR^m$.
\end{theorem}

In the 1950s, this result was generalized by using $S_z$---the set of all points in
$\RR^m$ having $z$ as a nearest point in $S$. Namely, the following characterization
of convexity is valid.

\begin{theorem}[Phelps, \cite{P}]
Let $S$ be a closed set in $\RR^m$. This set is convex if and only if $S_z$ is a
closed cone in $\RR^m$ with vertex $z$.
\end{theorem}

The following characterization of convexity was obtained in 1976. It involves the
Steiner sym\-metrization of a set (see, for example, \cite[Definition~12.7]{Val}),
and its converse is also true; see the monograph \cite{E}.

\begin{theorem}[Falconer, \cite{Fa}]
Let $S \subset \RR^m$ be a compact set of positive Lebesgue measure. If the Steiner
symmetrizations of this set about all hyperplanes are convex, then $S$ is convex.
\end{theorem}

Further data on history of convexity can be found in an account presented by Fenchel
\cite{F} in 1983. However, it was not until the mid-1980s that first results on
analytic characterization of convex sets had been published; see \cite{AK} and
\cite{Par}. Almost 40 years past since then, but there is still no summarizing
survey on diverse results in this field which has numerous applications; in
particular, in the analysis of Hardy-type inequalities \cite{BEL}. In the present
paper, our aim is to fill in this gap, at least partially, by reviewing the
mentioned, now classical, results along with new ones that appeared during the past
few years; see \cite{K} and \cite{S}.

The paper's plan is as follows. Several criteria for the convexity of a planar
domain/closed set are considered in Section~2; in particular, a characterization of
half-planes in terms of a distance function is given. In Section~3, various results
are presented which are valid for arbitrary $m \geq 2$. In particular, the
characterization of convexity is considered which involves a bilinear form defined
on exterior, unit normals on~$\partial D$. It is valid for a bounded domain $D$ with
smooth boundary and its proof is based on the Crofton formula for the surface
measure $|\partial D|$. The proof of another characterization of convexity is based
on application of another bilinear form; namely, the norm in the Sobolev space
$H^{1/2} (\RR^m)$.

{\centering \section{The convexity of planar sets} }

Throughout this section $D$ will denote a proper subdomain of the Euclidean plane
$\RR^2$, and so the boundary $\partial D$ is not empty. Hence the distance function
with respect to $\partial D$ is defined; namely,
\begin{equation}
d (x, \partial D) = \inf_{y \in \partial D} |x-y| \ \ \mbox{or} \ d (x) \ \mbox{for
brevity} \, ; \label{d}
\end{equation}
here $x = (x_1, x_2) \in D$ and $|x-y| = \sqrt{(x_1 - y_1)^2 + (x_2 - y_2)^2}$ is
the Euclidean norm in $\RR^2$. Thus the disc
\begin{equation}
B_r (x) = \{ y \in \RR^2 : |y-x| < r \} \subset D \ \ \mbox{provided} \ r = d (x) .
\label{disc}
\end{equation}
It occurs that the distance function $d (\cdot, \partial D)$ is a convenient tool
for characterization of convexity.

\vspace{2mm}

{\bf 2.1. Convexity of $\mathbf D$ via the distance function.} The following result
has a simple proof.

\vspace{1mm}

\begin{theorem}[Armitage and Kuran, \cite{AK}]
If $d (\cdot, \partial D)$ is superharmonic in a domain $D \subset \RR^2$, then $D$
is convex.
\end{theorem}

\begin{proof}
Let us assume $D \subset \RR^2$ to be nonconvex and show that $d (\cdot, \partial
D)$ is not superharmonic in~$D$. According to the assumption, there exist a point
$y \in \partial D$, a closed half-plane $P$ with $y \in \partial P$ (without loss of
generality, we set $y = (0,0)$ and $P = \{ x \in \RR^2 : x_2 \geq 0 \}$), and $r >
0$ such that
\[ P \cap (\overline{B_r (y)} \setminus \{y\}) \subset D \, , \ \ \mbox{whereas} \ \
(\partial D \setminus \{y\}) \cap B_r (y) \subset \RR^2 \setminus P \, ;
\]
cf. \cite[Theorem 4.8]{Val}.

Then for $x \in B = B_{r/8} (x_0)$, where $x_0 = (0, r/4)$, we have that $d (x,
\partial D) > x_2$ provided $x_1 \neq 0$. Hence
\[ \int_B d (x, \partial D) \, \D x > \int_B x_2 \, \D x = \pi (r/8)^2 (r/4) =
\pi (r/8)^2 d (x_0, \partial D) \, ,
\]
and so the area mean-value inequality, guaranteeing the superharmonicity of $d
(\cdot, \partial D)$, is violated for~$B$.
\end{proof}

It should be emphasized that Theorem 2.1 is specifically a two-dimensional result,
because neither $D$ nor $\bar D$ need be convex in higher dimensions. The example
given by Armitage and Kuran \cite[Sect.~5]{AK} to confirm this assertion involves
the domain in $\RR^3$ bounded by a particular torus.

\vspace{2mm}

{\bf 2.2. Characterization of half-planes.} Presumably, the first question about an
analytic characterization of a convex set appeared as an exercise in the textbook
\cite{Fu}. It concerns a half-plane which is the maximal, so to speak, convex domain
in $\RR^2$. The exceptional analytic property of its distance function $d$ arises
from the following observation.

Let a point $x \in \RR^2$ belong to a half-plane which can be taken $\{ x \in \RR^2
: x_2 > 0 \}$, without loss of generality. It is obvious that the distance from $x$
to the $x_1$-axis bounding this half-plane\,---\,the length of the segment
orthogonal to the $x_1$-axis\,---\,is equal to $x_2$, and so is a harmonic function
in this domain. The converse assertion is far from being trivial to prove; it is
referred to as one of ``three secrets about harmonic functions'' in the interesting
note \cite{B}, where the characterization of half-planes is as follows.

\begin{theorem}
If $d (\cdot, \partial D)$ is harmonic in a domain $D \subset \RR^2$, then $D$ is a
half-plane.
\end{theorem}

\begin{proof}[A sketch of proof]
It is clear that \eqref{disc} is valid for every $x \in D$ and $D = d^{-1} (0,
+\infty)$. Fixing $x^* \in D$ arbitrarily, we denote by $y^*$ a point on $\partial D$
with $|y^* - x^*| = d (x^*)$. Without loss of generality, $y^*$ can be taken as the
origin, whereas $x^* = (0, x_2^*)$ with $x_2^* > 0$. Then $B_{x_2^*} (x^*) \subset
D$ and the fact that
\[ d (x) = x_2 \ \mbox{for every} \ x = (0, x_2) \ \mbox{with} \ x_2 \in (0, 2 x_2^*] 
\]
is a consequence of harmonicity and continuity of $d$. Hence the equality $d (x) =
x_2$ is valid in $B_{x_2^*} (x^*)$ because harmonic functions are real-analytic. In
the same way, we have that $B_{2 x_2^*} (x^*_1) \subset D$, where $x_1^* = (0, 2
x_2^*)$, and so
\[ d (x) = x_2 \ \mbox{for every} \ x = (0, x_2) \ \mbox{with} \ x_2 \in (0, 2^2 x_2^*] ,
\]
in which case $d (x) = x_2$ in $B_{2 x_2^*} (x^*_1)$. Iterating this procedure, we
obtain that $d (x) = x_2$ in the disc $B_{2^n x_2^*} (x^*_n) \subset D$, where
$x_n^* = (0, 2^n x_2^*)$.

It is clear that the union of discs $\cup_{n \in \{1,2,\dots\}} B_{2^n x_2^*}
(x^*_n)$ coincides with the half-plane 
\[ P_0 = \{ x \in \RR^2 : x_2 > 0 \} \ \ \mbox{and} \ d (x) = x_2 \ \mbox{in} \ P_0 ,
\]
because this equality is valid in each disc. Since the latter is a subset of $D$, we
have that $P_0 \subset D$. On the other hand, $d$ is superharmonic in $D$, and so
this domain is convex by Theorem~2.1. But a convex domain can contain a half-plane
only coinciding with it; thus, $D = P_0$.
\end{proof}


{\bf 2.3. Convexity of $\mathbf D$ via solutions to the modified Helmholtz
equation.} It occurs that it is possible to describe the domain's convexity by means
of these solutions (labelled panharmonic functions in \cite{D}) by imposing some
conditions on them. Let $\partial D$ be bounded and sufficiently smooth, say $C^2$;
then the Dirichlet problem for the modified Helmholtz equation
\begin{equation}
\nabla^2 v - \mu^2 v = 0 \ \ \mbox{for} \ x \in D, \quad \mu \in \RR_+ ,
\label{MHh}
\end{equation}
has a unique solution satisfying the boundary condition
\begin{equation}
v (x) = 1 \quad \mbox{for} \ x \in \partial D ; \label{bc}
\end{equation}
here and below $\nabla = (\partial_1, \partial_2)$ is the gradient operator,
$\partial_i = \partial / \partial x_i$. It has long been known (see, for example,
\cite{H}) that a solution of this problem, say $v (x, \mu)$, satisfies the
inequalities
\begin{equation}
0 < v (x, \mu) \leq 1 \ \ \mbox{for all} \ x \in \bar D \ \mbox{and any} \ \mu > 0.
\label{ineq}
\end{equation}
In this connection, see the recent survey \cite[Sect.~2.1]{AN}, where a detailed
account of results concerning the strong maximum principle for elliptic operators is
presented.

Let us show that problem's solutions corresponding to large values of $\mu$
characterize the domain $D$ in the following way.

\begin{theorem}
Let $D \subset \RR^2$ be a bounded domain with smooth boundary. If the condition
\begin{equation}
|\nabla v (x, \mu)| \leq \mu v (x, \mu) , \ \ x \in D , \label{cond}
\end{equation}
holds for solutions of problem \eqref{MHh}, \eqref{bc} with all large $\mu > 0$, then
$D$ is convex.
\end{theorem}

Before proving this theorem, let us illustrate it with two elementary examples. In
the first one, the domain $D$ is the half-plane $P_0$, in which case it is natural
to require that $v (x, \mu) \to 0$ as $x_2 \to \infty$. Duffin \cite{D} established
that $v (x, \mu) = \E^{- \mu x_2}$ solves problem \eqref{MHh}, \eqref{bc}
complemented by the latter condition; see his proof of Theorem 5. Condition
\eqref{cond} is fulfilled for this function for all $\mu > 0$, and so Theorem~2.3
implies the obvious fact that $D$ is convex. Notice that
\[ - \mu^{-1} \log v (x, \mu) = x_2 = d (x, \partial D)
\]
in agreement with the assertion of Theorem~2.4 formulated below.

In the second example, the domain $D$ is a disc, and so convex. As usual, $I_0$
stands for the modified Bessel function of the first kind (see, for example,
\cite[p.~111]{D}). If $a \in (0, 1)$, then $v (x, \mu) = a I_0 (\mu |x|)$ solves
problem \eqref{MHh}, \eqref{bc} in the disc $B_r (0) = \{ y \in \RR^2 : |y| < r \}$
provided $r > 0$ is such that $a I_0 (\mu r) = 1$.

It is clear that properties of the function $I_0$ (see \cite[sect.~9.6]{AS}) yield
condition \eqref{cond} for $v (\cdot, \mu)$; indeed, it takes the form $I_1 (\mu |x|)
< I_0 (\mu |x|)$, because
\[ |\nabla v (x, \mu)| = \mu a I_1 (\mu |x|)
\]
in view of formula \cite[9.6.27]{AS}.

\vspace{1mm}

Now, we formulate two assertions essential for proving Theorem~2.3; they follow from
results obtained in the classical Varadhan's article \cite{V}.

\begin{theorem}[\cite{V}, p.~434]
Let $D \subset \RR^2$ be a bounded domain with smooth boundary. If $v (x, \mu)$ is a
solution of the Dirichlet problem \eqref{MHh}, \eqref{bc}, then 
\begin{equation}
- \mu^{-1} \log v (x, \mu) \to d (x, \partial D) \quad as \ \mu \to \infty
\label{lim}
\end{equation}
uniformly on $\bar D$.
\end{theorem}

The second assertion is a corollary of Theorem~3.6, \cite{V}.

\begin{proposition}
Let $v (x, \mu)$ be a solution of the Dirichlet problem \eqref{MHh}, \eqref{bc} in a
bounded domain $D \subset \RR^2$ with smooth boundary. Then for every $\rho \in (0,
1/2)$, there exists a constant $C_\rho > 1$ such that the estimate
\[ v (x, \mu) \leq C_\rho \exp \{ - \mu \, (1 - \rho) \, d (x, \partial D) \} \, ,
\quad x \in \bar D ,
\]
is valid for all $\mu > 0$.
\end{proposition}

\begin{proof}[Proof of Theorem 2.3.] It is sufficient to show that conditions
\eqref{cond} imply that $d$ is superharmonic in $D$; indeed, this allows us to apply
Theorem 2.1, thus demonstrating the convexity of $D$.

For a nonvanishing $u \in C^2 (D)$, one obtains by a straightforward calculation
that
\[ \nabla^2 (\log u) = - \frac{|\nabla u|^2}{u^2} + \frac{\nabla^2 u}{u} \quad
\mbox{in} \ D.
\]
Therefore, if  $v (\cdot, \mu) > 0$ is a solution of the Dirichlet problem
\eqref{MHh}, \eqref{bc} with large $\mu > 0$, then 
\[ - \nabla^2 (\log v) = - \mu^2 + |\nabla v|^2 / v^2 < 0
\]
provided condition \eqref{cond} is valid. Hence, $- \mu^{-1} \log v (\cdot, \mu)$ is
superharmonic in $D$ for all large $\mu > 0$.

The next step is to demonstrate that passage to the limit in Theorem~2.4 preserves
superharmonicity. In view of inequality \eqref{ineq} and condition \eqref{cond},
Proposition~2.1 implies that
\begin{equation}
- \mu^{-1} \log v (x, \mu) \geq - \mu^{-1} \log C_\rho + (1 - \rho) \, d (x,
\partial D) \, , \quad x \in \bar D . \label{final}
\end{equation}
Averaging this over an arbitrary open disc $B_r (x)$ such that $\overline{B_r (x)}
\subset D$, we obtain
\[ - \mu^{-1} \log v (x, \mu) \geq - \mu^{-1} \log C_\rho + \frac{1 - \rho}{\pi r^2}
\int_{B_r (x)} d (y, \partial D) \, \D y \, .
\]
Indeed, the function on the left-hand side of \eqref{final} is superharmonic and the
first term on the right is harmonic. Letting $\mu \to \infty$ first and then $\rho
\to +0$, we see that the last inequality turns into
\[ d (x, \partial D) \geq \frac{1}{\pi r^2} \int_{B_r (x)} d (y, \partial D) \, \D y
\, , \quad x \in D ,
\]
in view of \eqref{lim}.

Since the obtained inequality is valid for all $r > 0$ such that $\overline{B_r (x)}
\subset D$, the distance function $d (\cdot, \partial D)$ is superharmonic in $D$.
Then Theorem~2.1 yields the assertion of Theorem~2.3, thus completing its proof.
\end{proof}

Many assertions about convexity (see, in particular, Theorems~1.1 and 1.2) are
formulated in the following form: ``a set is convex if and only if some property is
fulfilled''. Therefore, it is interesting to ascertain whether the converse of
Theorem~2.3 is true.

\begin{conjecture}
Let $D \subset \RR^2$ be a bounded, convex domain with smooth boundary. Then
inequality \eqref{cond} is fulfilled for solutions of problem \eqref{MHh},
\eqref{bc} provided $\mu$ is sufficiently large.
\end{conjecture}

It is not unlikely that this assertion can be true for every $m \geq 2$.

\vspace{2mm}

{\bf 2.4. A local version of Theorem 2.1.} Let $F \neq \emptyset$ denote a proper,
closed subset of $\RR^2$; for every $x \in \RR^2$ the distance function from $F$ is
\begin{equation}
d (x, F) = \inf_{y \in F} |x-y| \, , \label{df}
\end{equation}
which is analogous to \eqref{d}. It occurs that this function characterizes the
convexity of $F$ in the same way as the distance function $d (\cdot, \partial D)$
does it for the domain $D$; namely, we have

\begin{theorem}[Parker, \cite{Par}]
Let $D \subset \RR^2$ be a domain such that $F \subset D$. Then the distance function $d
(\cdot, F)$ is subharmonic in $D$ if and only if $F$ is convex.
\end{theorem}

The proof is substantially more complicated than that of Armitage and Kuran's
result. Indeed, along with some simple auxiliary assertions it requires applying
the Krein--Milman theorem (see \cite[Theorem~11.5]{V}), as well as the following
local Motzkin-type theorem.

\begin{theorem}[Parker, \cite{Par}]
Let $F \neq \emptyset$ be proper, closed subset of\/ $\RR^m$, $m \geq 2$. If for
$x^* \in \partial F$ and some $r > 0$ every point in $B_r (x^*)$ has a unique
nearest point in $F$, then there exists an open ball of radius $r$ which touches $F$
at $x^*$; that is, this ball has no common points with $F$, but $x^*$ belongs to the
ball's boundary.
\end{theorem}
 
There is a very simple counterexample to Theorem~2.5 in three dimensions, and it can
be easily extended to higher dimensions. Indeed; let $(\rho, \theta, z)$ be
cylindrical coordinates in $\RR^3$, and let
\[ F = \{ \rho = 1 , \, z = 0 \} ;
\]
that is, $F$ is the unit circumference in the plane $\{ z = 0 \}$, which is
obviously not convex. Calculating the Laplacian of the axisymmetric distance
function
\[ d (\rho, \theta, z, F) = \sqrt{(\rho - 1)^2 + z^2} \, ,
\]
one immediately gets
\[ d_{\rho \rho} + \rho^{-1} d_\rho + d_{zz} = (2 \rho - 1) / (\rho d) \ \ \mbox{for}
\ \rho \neq 0 .
\]
Hence $d (\cdot, F)$ is subharmonic in the domain $\{ \rho > 1/2 , z \in \RR \}$
containing the nonconvex set $F$, which demonstrates that Theorem~2.5 is not valid
for domains of dimensions higher than two.

{\centering \section{The convexity of a set in arbitrary dimension} }

\noindent At the beginning of this section, we consider results similar to those
presented in Sections~2.1 and 2.4. Thereupon, an inequality characterizing the
convexity of a bounded domain $D$ with smooth boundary is described; it involves a
bilinear form defined on exterior, unit normals on~$\partial D$.

\vspace{2mm}

{\bf 3.1. The convexity via the subharmonicity of the distance function.} Let $F \neq
\emptyset$ be a pro\-per, closed subset of $\RR^m$, $m \geq 2$, then the distance $d
(x, F)$ from $F$ is defined by \eqref{df} for every $x = (x_1, \dots, x_m) \in
\RR^m$. The following assertion extends Theorem~2.5 to higher dimensions.

\vspace{1mm}

\begin{theorem}[Armitage and Kuran, \cite{AK}]
Let the closed set $F \subset \RR^m$, $m \geq 2$, be proper. Then the distance
function $d (\cdot, F)$ is subharmonic in $\RR^m \setminus F$ if and only if $F$ is
convex.
\end{theorem}

\vspace{-5mm}
  
\begin{proof}[A sketch of proof]
If $F$ is convex, then the signed distance $d (\cdot, H)$ (it is measured from every
its support hyperplane $H$ so that it is negative in the interior of $F$) is
harmonic in $\RR^m$; cf. Section~2.2. Moreover, we have that $d (\cdot, F) = \sup_H
d (\cdot, H)$, and so this function is subharmonic in $\RR^m \setminus F$, because
$d (\cdot, H)$ is harmonic.

To prove the ``only if'' part of theorem's assertion, let us assume that $F$ is
nonconvex. Then according to Theorem~1.1, there are two distinct points in $F$, say
$y_1$ and $y_2$, equidistant from some point in $\RR^m \setminus F$; without loss of
generality, the origin can be taken as this point, and so $|y_1| = |y_2| > 0$.
Putting
\[ v (x) = \min \{ |x - y_1| , \, |x - y_2| \} \ \ \mbox{for every} \ x \in \RR^m ,
\]
it is easy to show that there exists $r_* > 0$ such that
\[ v (0) > M (v, r) \ \ \mbox{for every} \ r \in (0, r_*) .
\]
Here $M (v, r)$ stands for the mean value of $v$ over the sphere $\{ x \in \RR^m :
|x| = r \}$. Since
\[ v (x) \geq d (x, F) \ \ \mbox{for every} \ x \in \RR^m \setminus F 
\]
with equality valid at the origin, we see that
\[ d (0, F) = v (0) > M (v, r) \geq M (d (\cdot, F), r) \, ,
\]
and so $d (\cdot, F)$ is not subharmonic in $\RR^m \setminus F$.
\end{proof}

{\bf 3.2. Characterization of convexity via an inequality for a bilinear form.} The
recent preprint \cite{S} begins with the following heuristic observation.

Let $D \subset \RR^m$, $m \geq 2$, be a bounded $C^1$-domain. If two points $x, y
\in \partial D$ are close, then the normals $n_x, n_y$ at these points are almost
parallel to each other and both of them are roughly orthogonal to $x-y$. However,
the normal turns quickly in a region, where the curvature is large, but along with
these regions there are ``flatter'' regions on the boundary of a convex domain.
Therefore, it might all average out in this case.

The quantitative result based on this argument (it can be interpreted as a global
conservation law for convex domains) is as follows.

\begin{theorem}[Steinerberger, \cite{S}]
Let $D \subset \RR^m$, $m \geq 2$, be a bounded $C^1$-domain. Then there exists a
constant $C_m > 0$ such that
\begin{equation}
\int_{\partial D \times \partial D} \frac{| (n_x, y-x) (y-x, n_y) |}{|x-y|^{m+1}} \,
\D S_x \D S_y \geq C_m |\partial D| \label{s}
\end{equation}
with equality taking place if and only if $D$ is convex.
\end{theorem}

Here $(\cdot, \cdot)$ is the inner product in $\RR^m$, whereas $\D S$ is the surface
measure on $\partial D$ and $|\partial D|$ is its total measure. If $D$ is the unit
ball, then the expression on the left-hand side of \eqref{s} reduces to
\[ \frac{\sqrt 2 \, (\pi / 2)^{m/2}}{\Gamma (m/2)} \int_{S^{m-1}} [1 - (x, w)]^{-(m-3)/2}
\D S_x \, , \quad w = (1, 0, \dots, 0) \in S^{m-1} ,
\]
which yields that $C_m = 2^{-1} \int_{S^{m-1}} |x_1| \, \D S$; in particular, $C_2 =
2$ and $C_3 = \pi$.

The proof of Theorem~3.2 given in the preprint \cite{S} relies essentially on
probabalistic technique; moreover, two assertions are crucial for this proof. The
first one is the Crofton formula for a rectifiable hypersurface in $\RR^m$ (see, for
example, \cite{Sa}); for our purpose it takes the form:
\begin{equation}
|\partial D| = \alpha_m \int_L N_{\partial D} (\ell) \, \D \varphi \, \D p \, ,
\quad \alpha_m = \frac{\Gamma ([m+1]/2)}{2 \, \pi^{(m-1)/2}} \, , \label{Cf}
\end{equation}
Here $L$ is the set of all oriented lines in $\RR^m$, the kinematic measure on
$S^{m-1} \times (0, \infty)$ is denoted $\D \varphi \, \D p$, and $N_{\partial D}
(\ell)$ stands for the number of transversal crossings of $\partial D$ by $\ell \in
L$. The second assertion is the following.

\vspace{-1mm}

\begin{lemma}
Let $D \subset \RR^m$, $m \geq 2$, be a bounded $C^1$-domain. Almost every (with
respect to the kinematic measure) line $\ell \in L$ either has no common points with
$\partial D$ or intersects it transversally at two points exactly if and only if $D$
is convex.
\end{lemma}

\vspace{-1mm}

As a consequence of this lemma and formula \eqref{Cf}, one obtains the inequality
\[ |\partial D| = \alpha_m \int_L N_{\partial D} (\ell) \, \D \varphi \, \D p \leq
\frac{\alpha_m}{2} \int_L \big[ N_{\partial D} (\ell) \big]^2 \, \D \varphi \, \D p
\, ,
\]
where equality takes place if and only if $D$ is convex. It occurs that \eqref{s}
follows from this inequality through a chain of tricky manipulations. A slight
modification of these considerations leads to the following.

\vspace{-1mm}

\begin{corollary}
Let $D \subset \RR^m$, $m \geq 2$, be a bounded, convex $C^1$-domain. Then
\[ \int_{\partial D} \frac{| (n_x, y-x) (y-x, n_y) |}{|x-y|^{m+1}} \, \D S_y = C_m
\quad \mbox{for every} \ x \in \partial D ;
\]
here $C_m$ is the same as in \eqref{s}.
\end{corollary}

\vspace{-1mm}

Since the Crofton formula is known to be true in a very general setting (see
\cite{Sa}), some conditions imposed on $D$ in this section can be relaxed. Indeed,
an analogue of inequality \eqref{s} was recently obtained \cite{Bu} under the
assumption that $D$ is just a set of finite perimeter. We recall that $D$ is a set
of this kind provided the distributional gradient $\nabla \chi_D$ of its
characteristic function $\chi_D$ defines a finite measure on $\RR^m$; that is,
$\int_{\RR^m} |\nabla \chi_D| < \infty$. (An exposition of the theory of these sets
is given in the monograph \cite{Ma}.) It occurs that the generalized inequality (see
\cite[formula (1.3)]{Bu}) involves the so-called reduced boundary $\partial^* D$
defined for a set of finite perimeter; in particular, integration over this
boundary is used instead of the usual domain's boundary $\partial D$.

Another analytic characterization of convexity was obtained by Figalli and Jerison
\cite{FJ} also for a set $F$ of finite perimeter. Namely, the authors considered
the function
\[ \{ x \in \RR^m : (x, v) = 0 \} \ni u \mapsto W_v (u) = \int_\RR \chi_F (u + t v) \, 
\D t \in \RR \, , \ \ \mbox{where} \ v \in \RR^m \ \mbox{and} \ |v| = 1 ,
\]
and proved the following.

\begin{theorem}
Let $F \subset \RR^m$, $m \geq 2$, be a bounded set of finite perimeter. If the
function $W_v$ is log-concave for almost every $v$ with respect to the
$(m-1)$-dimensional Hausdorff measure, then $F$ is convex (up to a set of measure
zero).
\end{theorem}

We recall that $W_v$ is log-concave if it has the form $\E^{-V}$, where $V$ is a
convex function mapping $\{ x \in \RR^m : (x, v) = 0 \}$ to $\RR \cup \{+\infty\}$.
Thus this theorem improves Theorem~1.3, according to which $W_v$ is required to be
concave to guarantee that $F$ is convex. The converse of Theorem~3.3 is true as
well. Indeed, if $F$ is convex, then the Brunn--Minkowski inequality \cite{G}
implies that $W_v$ is concave which is even stronger than log-concave.

The proof of Theorem~3.3 is based on measuring the perimeter of a set through the
standard norm in the Sobolev space $H^{1/2} (\RR^m)$, and so the assumption that $F$
has finite perimeter is essential. However, the authors point out that the theorem
could be true without it. It is also worth noticing that the norm in $H^{1/2}
(\RR^m)$ is a bilinear form like the left-hand side in \eqref{s}.

\vspace{2mm}

\noindent {\bf Acknowledgement.} The author thanks Stefan Steinerberger, whose
comments about the original manuscript helped him to improve the presentation.

\vspace{-12mm}

\renewcommand{\refname}{
\begin{center}{\Large\bf References}
\end{center}}
\makeatletter
\renewcommand{\@biblabel}[1]{#1.\hfill}

\end{document}